\UseRawInputEncoding
\documentclass[12pt]{amsart}
\usepackage{setspace, amsmath, amsthm, amssymb, amsfonts, amscd, epic, graphicx, ulem, dsfont}
\usepackage[T1]{fontenc}
\usepackage{multirow}
\usepackage{bbm}
\usepackage{enumerate}

\makeatletter \@namedef{subjclassname@2010}{
  \textup{2020} Mathematics Subject Classification}
\makeatother

\newtheorem{thm}{Theorem}[section]

\newtheorem{pro}[thm]{Proposition}

\theoremstyle{remark}
\newtheorem*{rema}{Remark}

\theoremstyle{definition}

\newtheorem{exa}[thm]{\textbf{Example}}

\newcommand{\ran}{\text{\rm{ran}}}

\newcommand{\C}{\mathbb{C}}

\begin{document}

\title[Adjoint of the unbounded operators $TT^*$ and $T^*T$]{Some properties of the adjoint of the unbounded operators $TT^*$ and $T^*T$}
\author[M. H. Mortad]{Mohammed Hichem Mortad}

\date{}


\address{Department of
Mathematics, University of Oran 1, Ahmed Ben Bella, B.P. 1524, El
Menouar, Oran 31000, Algeria.}

\email{mhmortad@gmail.com, mortad.hichem@univ-oran1.dz.}

\begin{abstract}
In this note, we mainly investigate the validity of the identities
$(TT^*)^*=TT^*$ and $(T^*T)^*=T^*T$, where $T$ is a densely defined
closable (or symmetric) operator.
\end{abstract}

\maketitle

\section{Introduction}

First, we assume readers have some familiarity with linear bounded
and unbounded operators on Hilbert spaces. Some useful references
are \cite{Mortad-Oper-TH-BOOK-WSPC}, \cite{Mortad-cex-BOOK}, and
\cite{SCHMUDG-book-2012}. Definitions and notations follow those in
\cite{SCHMUDG-book-2012}.

Let $T$ be a densely defined closed operator. One of the most
fundamental properties in unbounded operator theory is the fact that
$T^*T$ (and also $TT^*$) is a densely defined, self-adjoint and
positive operator. This is a very well-known von-Neumann's theorem.
In fact, von-Neumann's result may be obtained via the so-called
Nelson's trick (cf. \cite{Thaller-Dirac-EQuation-Contains Nelson
trick}). In particular,
\[(TT^*)^*=TT^*\text{ and } (T^*T)^*=T^*T.\]
This result then enables us to define the very important notion of
the modulus of an operator which, and as it is known, intervenes in
the definition of the polar decomposition of an operator. Operator
theorists are also well aware of other uses of the above result.
Some related results may be found in \cite{Boucif-Dehimi-Mortad},
\cite{Gesztesy et al. ANNALI. MATH. AA*}, \cite{Mortad-cex-BOOK},
\cite{RS2}, and \cite{sebestyen-tarcsay-TT* has an extension}.
Notice in the end that the above two equations are particular cases
of a more general and interesting problem in unbounded operator
theory, namely: When does one have
\[(AB)^*=B^*A^*\]
where $A$, $B$ and $AB$ are all densely defined operators? Some very
related papers are \cite{Azizov-Dijksma-closedness-prod-ADJ},
\cite{CG}, \cite{Gustafson-Mortad-I}, \cite{Sebestyen-Tarcsay-adj
sum and product}, and certain references therein.

Recently, Z. Sebestyén and Zs. Tarcsay have discovered that if
$TT^*$ and $T^*T$ are both self-adjoint, then $T$ must be closed
(see \cite{Sebestyen-Tarcsay-TT* von Neumann T closed}). Then F.
Gesztesy and K. Schm\"{u}dgen provided in
\cite{Gesztesy-Schmudgen-AA*} a simpler proof based on a technique
using matrices of unbounded operators. Notice that
Gesztesy-Schm\"{u}dgen's proof only work for complex Hilbert spaces
while the original proof by  Z. Sebestyén and Zs. Tarcsay works for
real Hilbert spaces just as good. To end this remark, Z. Sebestyén
and Zs. Tarcsay also gave a proof of their result using block
operator matrices as well as some other results (see e.g. Theorem
8.1 in \cite{Sebestyen-Tarcsay-LMA-2019}). Myself, jointly with S.
Dehimi, generalized the above results in
\cite{Dehimi-Mortad-squares-polynomials}.

In this paper, we investigate the validity of the identities
$(TT^*)^*=TT^*$ and $(T^*T)^*=T^*T$ where $T$ is a densely defined
closable (or symmetric) operator.

\section{Main Results}

\begin{thm}\label{ADJ AA* A unclosed THM}Let $T$ be a densely defined
operator. If $TT^*$ is densely defined and $\sigma(TT^*)\neq \C$,
then
\[(TT^*)^*=\overline{T}T^*=TT^*.\]
In particular, $TT^*$ is self-adjoint.
\end{thm}

\begin{proof}Clearly
\[\overline{T}T^*\subset (TT^*)^*.\]
Since $\sigma(TT^*)\neq \C$, consider a complex number $\lambda$
such that $TT^*-\lambda I$ is (boundedly) invertible. Then
\[(TT^*-\lambda I)^*=(TT^*)^*-\overline{\lambda} I\]
remains invertible.

On the other hand, since $\overline{T}T^*$ is self-adjoint,
$\overline{T}T^*-\overline{\lambda} I$ too is (boundedly)
invertible. Therefore
\[\overline{T}T^*\subset (TT^*)^*\Longrightarrow \overline{T}T^*-\overline{\lambda} I\subset (TT^*)^*-\overline{\lambda} I.\]
Since e.g. $\overline{T}T^*-\overline{\lambda} I$ is surjective and
$(TT^*)^*-\overline{\lambda} I$ is injective, Lemma 1.3 in
\cite{SCHMUDG-book-2012} gives $\overline{T}T^*-\overline{\lambda}
I=(TT^*)^*-\overline{\lambda} I$ or merely $\overline{T}T^*=
(TT^*)^*$.

To show the other equality, we may reason as above by using
$TT^*\subset \overline{T}T^*$. Alternatively, here is a different
approach: Since $\overline{T}T^*$ is self-adjoint, we have
\[\overline{T}T^*=(\overline{T}T^*)^*=(TT^*)^{**}=\overline{TT^*}=TT^*\]
where the closedness of $TT^*$ is obtained from $\sigma(TT^*)\neq
\C$.
\end{proof}

Mutatis mutandis, the following may also be shown.

\begin{pro}\label{ADJ A*A A unclosed pro}Let $A$ be a densely defined
operator. If $T^*T$ is densely defined and $\sigma(T^*T)\neq \C$,
then
\[(T^*T)^*=T^*\overline{T}=T^*T.\]
In particular, $T^*T$ is self-adjoint.
\end{pro}

\begin{rema}
When $\sigma(TT^*)\neq\C$ and $\sigma(T^*T)\neq \C$, then both
$TT^*$ and $T^*T$ are densely defined, and besides
\[(TT^*)^*=TT^*\text{ and } (T^*T)^*=T^*T.\]
This may be consulted in \cite{Hardt-Konstantinov-Spectrum-product}
and \cite{Hardt-Mennicken-OP-Th-ADv-APP}.
\end{rema}

It is natural to ask whether the identity $(TT^*)^*=TT^*$, which
holds for densely defined closed operators $T$, still holds for
closable operators.  The following example answers this question in
the negative.

\begin{exa}\label{Tarcsay example (TT*) neq T**T}Let $A$ be a densely defined positive operator in a Hilbert space $H$ such that it is
not essentially self-adjoint. Assume further that ran$A$ is dense in
$H$. Define an inner product space on $\ran A$ by
\[\langle Ax,Ay\rangle _{A}=\langle Ax,y\rangle,~x,y\in D(A).\]

Denote the completion of this pre-Hilbert space by $H_A$. Define the
canonical embedding operator $T:H_A\supseteq \text{ran}A\to H$ by
\[T(Ax):=Ax,~x\in D(A).\]
It may then be shown that $D(A)\subset D(T^*)$ and
\[T^*x=Ax\in H_A,~x\in D(A).\]
In particular, $T:H_A\supset \ran A\to H$ is a densely defined and
closable linear operator. Besides, $T^{**}T^*$ is a self-adjoint
positive extension of $A$.

On the other hand, notice that
\[\ker T^*=(\ran T)^{\perp}=(\ran A)^{\perp}=\{0\}.\]
In other words, $T^*$ is one-to-one. So, $T^*y\in\ran A$ (=$D(T)$),
where $y\in D(T^*)$, implies $y\in D(A)$. Accordingly,
\[D(TT^*)=\{y\in D(T^*):T^*y\in D(T)\}=D(A).\]
Thus, $TT^*=A$. This signifies that
\[(TT^*)=A^*\neq T^{**}T^*=\overline{T}T^*\]
for the latter operator is self-adjoint, whilst the former is not.
\end{exa}

The preceding example may be beefed up to obtain a stronger
counterexample.

\begin{exa}Let $T$ be a closable densely defined operator such that
$(TT^*)^*\neq \overline{T}T^*$ (e.g. as in Example \ref{Tarcsay
example (TT*) neq T**T}), then set
\[S=\left(
      \begin{array}{cc}
        0 & T \\
        T^* & 0 \\
      \end{array}
    \right)
\]
with $D(S)=D(T)\oplus D(T^*)$. So $S$ is densely defined, and since
$S^*=\left(
      \begin{array}{cc}
        0 & \overline{T} \\
        T^* & 0 \\
      \end{array}
    \right)$, it is seen that $S$ is symmetric.

Now,
\[SS^*=\left(
         \begin{array}{cc}
           TT^* & 0 \\
           0 & T^*\overline{T} \\
         \end{array}
       \right)\text{ and }\overline{S}S^*=\left(
         \begin{array}{cc}
           \overline{T}T^* & 0 \\
           0 & T^*\overline{T} \\
         \end{array}
       \right).\]
But
\[(SS^*)^*=\left(
         \begin{array}{cc}
           (TT^*)^* & 0 \\
           0 & (T^*\overline{T})^* \\
         \end{array}
       \right)=\left(
         \begin{array}{cc}
           (TT^*)^* & 0 \\
           0 & T^*\overline{T} \\
         \end{array}
       \right)
\]
because $T^*\overline{T}$ is self-adjoint. Since we already know
that $(TT^*)^*\neq \overline{T}T^*$, it ensues that $(SS^*)^*\neq
\overline{S}S^*$ and yet $S$ is densely defined and symmetric.
\end{exa}

We have been assuming that $TT^*$ is densely defined as, in general,
$D(TT^*)$ could be non-dense. For instance, in
\cite{Mortad-TRIVIALITY POWERS DOMAINS} (or \cite{Mortad-cex-BOOK}),
we have found an example of a densely defined $T$ which obeys
\[D(T^2)=D(T^*)=D(TT^*)=D(T^*T)=\{0\}.\]

Obviously such an operator $T$ cannot be closable. So, does the
closability of $T$ suffice to make $TT^*$ densely defined? The
answer is negative even when $T$ is symmetric. This is seen next.

\begin{exa}
Let $E$ be a dense linear proper subspace of $H$, let $u$ be a
non-zero element in $H$ but not in $E$, and define $T$ to be
projection on the 1-dimensional subspace spanned by $u$ with
$D(T)=E$. Then $T$ is a bounded, non-everywhere defined, unclosed,
and symmetric operator. Also, $T^*$ is the same projection, defined
on the entire $H$. Then $T^*x$ is in $E$ only if it is 0. So,
\[D(TT^*)=\{u\}^{\perp},\] which is not dense.
\end{exa}

\section{Acknowledgement}

The author wishes to thank Professor Zsigmond Tarcsay for Example
\ref{Tarcsay example (TT*) neq T**T} which was based on a
construction by Z. Sebestyén and J. Stochel (see
\cite{Sebestyen-Stochel-Restrictions-positive}).

\end{document}